\documentclass[%dvipdfmx, 
10pt,a4,leqno]{amsart}
\usepackage{amsthm}
\usepackage{amsmath}
\usepackage{amssymb}
\usepackage{amsfonts}
\usepackage{mathrsfs}
\usepackage{graphicx}
\usepackage{bbm}
\usepackage[%dvipdfmx,
colorlinks=true, citecolor=blue, linkcolor=blue, urlcolor=red]{hyperref}
\usepackage{color}
\usepackage[all]{xy}
\usepackage{comment}
\usepackage{here}
\usepackage{amscd}
\usepackage{caption}
\usepackage{listings}
\usepackage{setspace}
\usepackage{booktabs}
\allowdisplaybreaks

\newcommand{\mathsym}[1]{{}}
\newcommand{\unicode}[1]{{}}

\makeatletter

\@addtoreset{equation}{section}
\makeatother

\makeatletter
\@namedef{subjclassname@2020}{%
  \textup{2020} Mathematics Subject Classification}
\makeatother

\newtheoremstyle{my theoremstyle}
{1.0em}                    % Space above
    {1.0em}                    % Space below
    {\itshape}                   % Body font
    {}                           % Indent amount
    {\scshape}                   % Theorem head font
    {.}                          % Punctuation after theorem head
    {.5em}                       % Space after theorem head
    {}  % Theorem head spec (can be left empty, meaning ‘normal’)

\newtheoremstyle{dfn}
{1.0em}                    % Space above
    {1.0em}                    % Space below
    {}                   % Body font
    {}                           % Indent amount
    {\scshape}                   % Theorem head font
    {.}                          % Punctuation after theorem head
    {.5em}                       % Space after theorem head
    {}  % Theorem head spec (can be left empty, meaning ‘normal’)
    
\theoremstyle{my theoremstyle}
   \newtheorem{thm}{Theorem}[section]
   \newtheorem{lem}[thm]{Lemma}
   \newtheorem{prop}[thm]{Proposition}
   \newtheorem{cor}[thm]{Corollary}
   \newtheorem{conj}[thm]{Conjecture}
\theoremstyle{dfn}
   \newtheorem{dfn}[thm]{Definition}
   
\theoremstyle{remark}   
   
   \newtheorem{rmk}[thm]{{\scshape Remark}}

\newcommand{\Z}{\mathbb{Z}}
\newcommand{\Q}{\mathbb{Q}}

\newcommand{\C}{\mathbb{C}}
\newcommand{\F}{\overline{\mathbb{F}}_p}
\newcommand{\e}{\varepsilon}

\newcommand{\tw}{\widetilde{\omega}}

\numberwithin{equation}{section}
\newcommand{\ub}{\underline{b}}
\newcommand{\ua}{\underline{a}}

\newcommand{\lm}{\lambda}

\newcommand{\Spec}{\operatorname{Spec}}

\date{\today}

\begin{document}
\title[$p$-adic hypergeometric functions of logarithmic type]{On special values of generalized $p$-adic hypergeometric functions of logarithmic type}
\author{Yusuke Nemoto}
\date{\today}
\address{Graduate School of Science, Chiba University, 
Yayoicho 1-33, Inage, Chiba, 263-8522 Japan.}
\email{y-nemoto@waseda.jp}

\maketitle

\begin{abstract}
We introduce a new type of $p$-adic hypergeometric functions, which are generalizations of $p$-adic hypergeometric functions of logarithmic type 
defined by Asakura, and show that these functions satisfy the congruence relations similar to Asakura's. 
We also give numerical computations of the special values of these functions at $t=1$ and prove that these values are equal to zero under some conditions.  
\end{abstract}

\section{Introduction} 
Let $p$ be a prime and $s \geq 2$ be an integer.  
For an $s$-tuple $\ua=(a_1, \ldots, a_s) \in \Z_p^s$ and an $(s-1)$-tuple $\ub=(b_1, \ldots, b_{s-1}) \in (\Z_p \backslash \Z_{ \leq 0})^{s-1}$, we define the \textit{hypergeometric series} by  
$$F_{\ua, \ub}(t)={_{s}F_{s-1}}\left( 
\begin{matrix}
a_1, \cdots, a_s \\
b_1, \ldots b_{s-1}
\end{matrix}
; t
\right)
=
\sum_{n=0}^{\infty}\dfrac{(a_1)_n \cdots (a_s)_n}{(b_1)_n \cdots (b_{s-1})_n (1)_n}t^n. 
$$
Here, $(\alpha)_n=\alpha (\alpha+1) \cdots (\alpha +n-1)$ denotes the Pochhammer symbol. 
When $\ub=\underline{1}:=(1,\ldots, 1)$, $F_{\ua, \underline{1}}(t)$ belongs to $\Z_p[[t]]$ for all $\ua \in \Z_p^s$; otherwise it belongs to $\Q_p[[t]]$ in general.  
For $a \in \Z_p$, let $a'$ be the \textit{Dwork prime}, which is defined to be $a'=(a+l)/p$ where $l \in \{0, \ldots, p-1\}$ is the unique integer such that $a+l \equiv 0 \pmod{p}$. 
Let $a^{(i)}$ be the $i$th Dwork prime defined by $a^{(i)}=(a^{(i-1)})'$ and $a^{(0)}=a$.  
Put 
$$F^{(i)}_{\ua, \ub}(t):={_{s}F_{s-1}}\left( 
\begin{matrix}
a^{(i)}_1, \cdots, a^{(i)}_s \\
b^{(i)}_1, \ldots b^{(i)}_{s-1}
\end{matrix}
; t
\right). 
$$
We drop $\ub$ from the notation when $\ub = \underline{1}$. 
In this paper, we consider those $\ua$ and $\ub$ for which $F^{(i)}_{\ua, \ub}(t) \in \Z_p[[t]]$ for any $i \geq 0$ (see the conditions (i) and (ii) in Theorem \ref{theorem:1}). 
Then Dwork defines the \textit{$p$-adic hypergeometric function} by $\mathscr{F}^{\rm Dw}_{\ua, \ub}(t)=F_{\ua, \ub}(t)/F^{(1)}_{\ua, \ub}(t^p)$ and 
proves the congruence relation
\begin{align} \label{cong1}
\mathscr{F}_{\ua, \ub}^{\rm Dw}(t) \equiv \dfrac{[F_{\ua, \ub}(t)]_{<p^n}}{[F^{(1)}_{\ua, \ub}(t^p)]_{<p^n}} \pmod{p^n \Z_p[[t]]}
\end{align}
for any $n \geq 1$ (see Theorem \ref{theorem:1}). 
Here, for a power series $f(t)=\sum a_it^i$, we write $[f(t)]_{<n}=\sum_{i<n}a_it^i$ the truncated polynomial.  
Using \eqref{cong1}, we can define special values (cf. \cite[Corollary 2.3]{As2}), which relate to the unit roots of elliptic curves of Legendre type (see \cite[Theorem (8.1)]{Dw}).

Let $W=W(\overline{\mathbb{F}}_p)$ be the Witt ring of $\overline{\mathbb{F}}_p$ and $\sigma$ be the $p$th Frobenius on $W[[t]]$ defined by $\sigma(t)=ct^p$ ($c \in 1+pW$). 
Let $\psi_p(z)$ be the $p$-adic digamma function and $\gamma_p$ be the $p$-adic Euler constant defined in Section \ref{log_pHG}, and $\log \colon \C_p^* \to \C_p$ be the Iwasawa logarithmic function. 
Put 
$$G_{\ua}(t)=\sum_{i=1}^{s}\psi_p(a_i)+ s\gamma_p-p^{-1} \log(c) +\int_0^t (F_{\ua}(t)-F^{(1)}_{\ua}(t^{\sigma}))\dfrac{dt}{t}, $$
where $\int_0^t(-)\frac{dt}{t}$ is an operator which sends a power series $\sum_{n=1}^{\infty} a_nt^n$ to $\sum_{n=1}^{\infty} \frac{a_n}{n} t^n$. 
In the paper \cite{As}, Asakura defines a new type of the $p$-adic hypergeometric function by 
$$\mathscr{F}^{(\sigma)}_{\ua}(t)=\dfrac{G_{\ua}(t)}{F_{\ua}(t)} \in W[[t]], $$
which is called the \textit{$p$-adic hypergeometric function of logarithmic type}.   
If $a_i \not \in \Z_{\leq 0}$ for any $i$, Asakura \cite[Theorem 3.2]{As} proves a congruence relation of $\mathscr{F}^{(\sigma)}_{\ua}(t)$ similar to \eqref{cong1}. 
Thanks to this, we can define special values (see \cite[Corollary 3.4]{As}), which relate to the $p$-adic regulators of hypergeometric motives whose periods are described by $F_{\ua}(t)$ (e.g. elliptic curves of Legendre type, see \cite{As, As3}).

In this paper, we generalize the $p$-adic hypergeometric function of logarithmic type to the case $\ub \neq \underline{1}$ and prove the congruence relations. 
Suppose that $\ua$ and $\ub$ satisfy the conditions (i) and (ii) in Theorem \ref{theorem:1}, i.e. $F^{(i)}_{\ua, \ub}(t) \in \Z_p[[t]]$ and satisfies the congruence relation \eqref{cong1}.  
Put 
$$G_{\ua, \ub}(t)= \sum_{i=1}^{s}\psi_p(a_i)-\sum_{i=1}^{s-1}\psi_p(b_i) +\gamma_p-p^{-1} \log(c) +\int_0^t (F_{\ua, \ub}(t)-F_{\ua, \ub}^{(1)}(t^{\sigma}))\dfrac{dt}{t}$$
and we define the \textit{(generalized) $p$-adic hypergeometric function of logarithmic type} by 
\begin{align*} 
\mathscr{F}^{(\sigma)}_{\ua, \ub}(t)= \dfrac{G_{\ua, \ub}(t)}{F_{\ua, \ub}(t)}. 
\end{align*}
When $\ub=\underline{1}$, $\mathscr{F}^{(\sigma)}_{\ua, \ub}(t)$ agrees with $\mathscr{F}_{\ua}^{(\sigma)}(t)$ since $\psi_p(1)=-\gamma_p$ by \cite[Theorem 2.6]{As}.  
We prove $\mathscr{F}^{(\sigma)}_{\ua, \ub}(t) \in W[[t]]$ (see Lemma \ref{inW}) and the following congruence relation.  
\begin{thm} \label{main:1}
Suppose that $\ua$ and $\ub$ satisfy the conditions (i) and (ii) in Theorem \ref{theorem:1}, and $a_i \not \in \Z_{\leq 0}$ for all $i$. 
If $p$ is odd, then for all $n \geq 1$, we have 
\begin{align*} 
\mathscr{F}^{(\sigma)}_{\ua, \ub}(t)\equiv \dfrac{[G_{\ua, \ub}(t)]_{<p^n}}{[F_{\ua, \ub}(t)]_{<p^n}} \pmod{p^nW[[t]]}.  
\end{align*}
If $p=2$, the congruence above holds modulo $p^{n-1}W[[t]]$. 
\end{thm}

As a corollary of Theorem \ref{main:1}, 
we can prove that $\mathscr{F}^{(\sigma)}_{\ua, \ub}(t)$ is an analytic element in the following sense:
$$\mathscr{F}^{(\sigma)}_{\ua, \ub}(t) \in W \langle t, h_{\ua, \ub}(t)^{-1}\rangle:=\varprojlim_{n \geq 1} (W/p^nW[t, h_{\ua, \ub}(t)^{-1}]), \quad h_{\ua, \ub}(t)=\prod_{i=0}^N[F^{(i)}_{\ua, \ub}(t)]_{<p}$$
for some $N \gg 0$ 
(see Corollary \ref{maincor}). 
 Therefore, for $\alpha \in W$ such that $|h_{\ua, \ub}(\alpha)|_p=1$, we can define the special value at 
 $t=\alpha$ 
 by 
\begin{align} \label{value}
\mathscr{F}_{\ua, \ub}^{(\sigma)}(t)|_{t=\alpha} = \mathscr{F}_{\ua, \ub}^{(\sigma)}(\alpha)=\lim_{n \to \infty}  \left(\left. \dfrac{[G_{\ua, \ub}(t)]_{<p^n}}{[F_{\ua, \ub}(t)]_{<p^n} }\right|_{t=\alpha} \right). 
\end{align}

From now on, let $\sigma$ be the $p$th Frobenius on $W[[t]]$ defined by $\sigma(t)=t^p$. 
In this paper, we consider the special value at $t=1$ for $s=2$. 
For $\ua=(a_1, a_2)$, $\ub=(b)$, we write $\mathscr{F}_{\ua, \ub}^{(\sigma)}(t)$ as $\mathscr{F}_{a_1, a_2; b}^{(\sigma)}(t)$. 
We prove that the special value $\mathscr{F}_{\frac{i}N, \frac{j}N; \frac{k}N}^{(\sigma)}(1)$ is defined under the assumptions $N \mid p-1$ and $i+j \leq k$, and give numerical computations of $\mathscr{F}_{\frac{i}N, \frac{j}N; \frac{k}N}^{(\sigma)}(1) \pmod{p^4}$ for $N=2, 3, 4, 5, 6$ and $p=3, 5, 7, 11, 13$ (see Theorem \ref{main:2}).  

The numerical examples in Table \ref{numerical} suggest the following conjecture. 
\begin{conj} \label{conj:1}
Let $N$ be a positive integer and $p$ be a prime such that $N \mid p-1$. 
For integers $i, j \in \{1, \ldots, N-1\}$ such that $i+j \leq N$, 
we have 
$$\mathscr{F}_{\frac{i}{N}, \frac{j}{N};\frac{i+j}{N}}^{(\sigma)}(1)=0. $$
\end{conj} 
In this paper, we prove Conjecture \ref{conj:1} under a condition, namely: 
\begin{thm} \label{main:3}
If $i+j=N$,  
then Conjecture \ref{conj:1} is true. 
\end{thm}

We briefly give a sketch of the proofs of Theorems \ref{main:1} and \ref{main:3}.  
The proof of Theorem \ref{main:1} is as follows. 
First, we reduce it to the case $c=1$ (see Section \ref{reduce1}). 
Write $F_{\ua, \ub}(t)=\sum C_nt^n$, $F^{(1)}_{\ua, \ub}(t)=\sum C_n^{(1)}t^n$ and $G_{\ua, \ub}(t)=\sum D_nt^n$. 
It suffices to show that for any $m, n \in \Z_{\geq 0}$, 
\begin{align}
\sum_{i+j=m}(C_{i+p^n}D_j-C_iD_{j+p^n}) \equiv 0 \pmod{p^n}. \label{maincong} 
\end{align}
When $\ub=\underline{1}$, this is proved by using the congruence relations of $p$-adic integers 
 \begin{align}
\dfrac{C_{m_1}}{C^{(1)}_{\lfloor m_1/p \rfloor}} &\equiv \dfrac{C_{m_2}}{C^{(1)}_{\lfloor m_2/p \rfloor}} \pmod{p^n},  \label{C} \\
\dfrac{D_{m_1}}{C_{m_1}} &\equiv \dfrac{D_{m_2}}{C_{m_2}} \pmod{p^n} \label{D}, 
\end{align}
where $m_1 \equiv m_2 \pmod{p^n}$ (see \cite[Section 3]{As}).  
For general $\ub$, although we can prove \eqref{D} (see Lemma \ref{lemma:19}), $C_m/C^{(1)}_{\lfloor m/p \rfloor}$ is not a $p$-adic integer in general (see Lemma \ref{lemma:2}), hence we cannot use the same method as in loc. cit. for our proof. 
To avoid this problem, we need a different method using Lemma \ref{keylem}, which is a slight modification of the key lemma \cite[Lemma 3.12]{As}, proving \eqref{maincong} without using the congruence relation \eqref{C}.  
The method used in this paper can be applied to the case where $\ub=\underline{1}$, making it easier to prove the congruence relations than the original proof \cite[Section 3]{As}.  

The proof of Theorem \ref{main:3} is as follows.  
We consider a fibration $f \colon Y \to \mathbb{P}^1_W$ over $W$, where the affine model of the general fiber $f^{-1}(t)$ is given by
$$(1-x^N)(1-y^N)=t, $$
which is called the \textit{hypergeometric curve} introduced by Asakura and Otsubo \cite{AO2}. 
Put 
$S:=\Spec W[t, (t-t^2)^{-1}]$ and $X:=f^{-1}(S)$. 
Let $\xi$ be a Milonor $K_2$-symbol of $X$ defined in \eqref{xi}. 
Let $\sigma$ be the $p$th Frobenius on $W[t, (t-t^2)^{-1}]^{\dagger}$ (the ring of overconvergent power series) defined by $\sigma(t)=t^p$. 
Then Asakura \cite[Theorem 4.8]{As} proves that the image of $\xi$ under the $p$-adic regulator induced by $\sigma$ is expressed in terms of $\mathscr{F}_{\ua}^{(\sigma)}(t)$. 
On the other hand, we put $\lm=1-t$ and let $\tau$ be another $p$th Frobenius on $W[t, (t-t^2)^{-1}]^{\dagger}$ defined by $\tau(\lm)=\lm^{p}$. 
The key step of our proof is explicitly expressing the image of $\xi$ under the $p$-adic regulator induced by $\tau$ (see Theorem \ref{de}). 
By comparing the $p$-adic regulators using Lemma \ref{comp}, we obtain Theorem \ref{main:3}.  

The hypergeometric curve specializes at $t=1$ to the Fermat curve of degree $N$
$$X_N : x^N+y^N=1. $$
Let $K$ be the fractional field of $W$ and put $X_{N, K}:=X_N \times_{W} K$.  
In the category of pure motives over $K$ with $K$-coefficients, $h^1(X_{N, K})$ is decomposed into the direct sum of rank $1$ motives $X_N^{a, b}$.  
When $i + j < N$, the special value $\mathscr{F}^{(\sigma)}_{\frac{i}N, \frac{j}N; 1}(1)$ appears in the $p$-adic regulator of the motive $X_N^{i, j}$ (see \cite[Theorem 4.11]{As}), and conjecturally, does not vanish (see \cite[Conjecture 4.17]{As}).  
Theorem \ref{main:3} seems to be related to the fact that the motive $X_N^{i, N-i}$ is trivial. 
Similarly, we expect geometric interpretations of our function  $\mathscr{F}_{\ua, \ub}^{(\sigma)}(t)$ as follows: 
There is a so-called generalized Fermat curve (cf. \cite{Nova}) 
\begin{align} \label{generalFermat}
\left\{
\begin{array}{ll}
x^N+y^N=1  \\
x^N+z^N=t
\end{array}
\right.
\end{align}
whose complex periods are described by $F_{a_1, a_2; b}(t)$ for $b$ possibly different from $1$. Similar to the  case $b=1$ due to Asakura, we expect that the $p$-adic regulators of \eqref{generalFermat} are written 
in terms of the special values of our function $\mathscr{F}_{a_1, a_2; b}^{(\sigma)}(t)$. 
Also, we expect to prove Conjecture \ref{conj:1} by using these geometric interpretations.

This paper is constructed as follows. 
In Section \ref{Dwork_pHG}, we recall the Dwork $p$-adic hypergeometric function and its congruence relations.  
In Section \ref{log_pHG}, we recall the $p$-adic hypergeometric function of logarithmic type $\mathscr{F}_{\ua}^{(\sigma)}(t)$ and its congruence relations. 
 After that, we give the definition of the generalized $p$-adic hypergeometric function $\mathscr{F}^{(\sigma)}_{\ua, \ub}(t)$.  
In Section \ref{main_proof}, we give a proof of Theorem \ref{main:1}, i.e. the congruence relations of $\mathscr{F}^{(\sigma)}_{\ua, \ub}(t)$.  
In Section \ref{special_value}, we prove Theorems \ref{main:2} and \ref{main:3}. 
To prove Theorem \ref{main:3}, we also recall the relation between $p$-adic regulators of hypergeometric curves and $p$-adic hypergeometric functions of logarithmic type.

\subsection{Notations}
For a field $K$ and a positive integer $N$, let $\mu_N \subset K^*$ be the group of  $N$th roots of unity. 
For a power series $f(t)=\sum a_nt^n$, we write $f(t)_{<m}= \sum_{n<m}a_nt^n$ the truncated polynomial.

\section{Dwork's $p$-adic hypergeometric function} \label{Dwork_pHG}
In this section, we recall $p$-adic hypergeometric functions defined by Dwork for $\ub=\underline{1}$ in \cite{Dw} and for general $\ub$ in \cite{Dw2}. 

For $\alpha \in \Z_p$, let $\alpha'$ be the Dwork prime defined by $\alpha'=(\alpha+l)/p$, where $l \in \{0, 1, \ldots, p-1\}$ is the unique integer such that $\alpha+l \equiv 0 \pmod{p}$. To put it another way, if $p$-adic integer $\alpha$ has the $p$-adic expansion 
$$\alpha=-[\alpha]_0-[\alpha]_1p-[\alpha]_2p^2-\cdots -[\alpha]_np^n-\cdots$$
where $[\alpha]_i \in \{0,1, \cdots, p-1\}$, then $\alpha^{\prime}$ is defined by
$$\alpha^{\prime}=-[\alpha]_1-[\alpha]_2p-[\alpha]_3p^2-\cdots -[\alpha]_{n-1}p^n-\cdots.$$
Let $a^{(i)}$ be the $i$th Dwork prime defined by $a^{(i)}=(a^{(i-1)})^{\prime}$ with $a^{(0)}=a$. 
Put 
$$F_{\ua, \ub}(t)={_{s}F_{s-1}}\left( 
\begin{matrix}
a_1, \cdots, a_s \\
b_1, \ldots b_{s-1}
\end{matrix}
; t
\right), \quad F^{(i)}_{\ua, \ub}(t)={_{s}F_{s-1}}\left( 
\begin{matrix}
a_1^{(i)}, \cdots, a_s^{(i)} \\
b_1^{(i)}, \ldots b_{s-1}^{(i)}
\end{matrix}
; t
\right). $$
We drop $\ub$ from the notation when $\ub = \underline{1}$. 
Then Dwork \cite{Dw2} proves the following theorem.

\begin{thm}[{\cite[Theorem 1.1]{Dw2}}, cf. {\cite[Theorem 2.3]{Young}}]\label{theorem:1}
Let $\ua=(a_1, \ldots, a_s) \in (\mathbb{Q} \cap \mathbb{Z}_p)^s$, $\ub=(b_1, \ldots, b_{s-1}) \in (\mathbb{Q} \cap \mathbb{Z}_p\backslash \mathbb{Z}_{\leq 0})^{s-1}$ such that $b_j \neq 1$ for $1 \leq j \leq q$ and $b_j=1$ for $q < j \leq s-1$. Assume that the following conditions are satisfied: 
\begin{enumerate}
\item $|b^{(i)}_j|_p=1$ for all $i \geq 0$, $j=1,\ldots, q$,  
\item For each fixed $n$, supposing the indices are rearranged so that $[a_{1}]_n \leq \cdots \leq [a_{s}]_n$ and $[b_1]_n \leq \cdots \leq [b_{s-1}]_n$, 
we have 
$$  [a_{j+1}]_n < [b_j]_n \quad (j=1, \ldots, q).$$
\end{enumerate}
\noindent Then for any $i \geq 0$, we have $\displaystyle F^{(i)}_{\ua, \ub}(t) \in \mathbb{Z}_p[[t]]$ and the congruence relations such that
$$F_{\ua, \ub}(t) [F^{(1)}_{\ua, \ub}(t^p)]_{<p^n} \equiv F^{(1)}_{\ua, \ub}(t^p) [F_{\ua, \ub}(t)]_{<p^n}    \pmod{p^{l} \mathbb{Z}_p[t]} \quad  (n \geq l). $$
\end{thm}

Dwork defines the $p$-adic hypergeometric function by a ratio of hypergeometric series  
\begin{align*}
\mathscr{F}_{\ua, \ub}^{\rm Dw}(t)=\dfrac{F_{\ua, \ub}(t)}{F^{(1)}_{\ua, \ub}(t^p)}. 
\end{align*}
By Theorem \ref{theorem:1}, we have the congruence relations
\begin{align*}
\mathscr{F}_{\ua, \ub}^{\rm Dw}(t) \equiv \dfrac{[F_{\ua, \ub}(t)]_{<p^n}}{[F^{(1)}_{\ua, \ub}(t^p)]_{<p^n}} \pmod{p^n\mathbb{Z}_p[[t]]}. \end{align*}
Therefore, this function is $p$-adically analytic in the sense of Krasner, which means that this function defines an element of the Tate algebra (cf. \cite[Corollary 2.3]{As2}). 

\section{$p$-adic hypergeometric functions of logarithmic type} \label{log_pHG}
In the paper \cite{As}, Asakura defines another $p$-adic hypergeometric function for $\ub=\underline{1}$, which is called the $p$-adic hypergeometric function of logarithmic type, and proves congruence relations similar to Dwork's. 
In this section, we briefly recall his function and the congruence relations. 
After that, we define a $p$-adic hypergeometric function of logarithmic type for possibly different from $\underline{1}$.

\begin{dfn} \label{definition:3}
For $z \in \mathbb{Z}_p$, we define 
$$\widetilde{\psi}_p(z) = \displaystyle \lim_{n \in \mathbb{Z}_{>0}, n \rightarrow z} \sum _{1 \leq k <n, p \nmid  k}\dfrac{1}{k}.$$
Define the \textit{$p$-adic Euler constant} by 
$$\gamma_p=\displaystyle \lim_{s \to \infty}\dfrac{1}{p^s} \sum_{0 \leq j<p^s, p \nmid j }\log(j),$$
where $\log \colon \C^*_p \to \C_p$ is the \textit{Iwasawa logarithmic function}, which is characterized as a continuous homomorphism satisfying $\log(p)=0$ and 
$$\log(z)=-\displaystyle \sum_{n=1}^{\infty} \dfrac{(1-z)^n}{n}, \quad |z-1|_p<1.$$ 
We define the \textit{$p$-adic digamma function} by
$$\psi_p(z)=-\gamma_p +\widetilde{\psi}_p(z).$$
\end{dfn}
If $z \equiv z^{\prime} \pmod{p^s}$, then for $p \geq 3$ or $p=2$, $s \geq 2$, we have 
\begin{align}
\widetilde{\psi}_p(z)-\widetilde{\psi}_p(z^{\prime}) \equiv 0 \pmod {p^s}  \label{eq:6}
\end{align}
(see \cite[Lemma 2.5 and (2.13)]{As}), 
hence this function is a $p$-adic continuous function on $\mathbb{Z}_p$.

Let $\sigma$ be a $p$th Frobenius on $W[[t]]$ defined by $\sigma(t)=ct^p$ ($c \in 1+pW$), i.e.
$$\left(\sum_i a_i t^i \right)^{\sigma}=\sum_i a_i^Fc^it^{ip}, $$
where $F$ is the Frobenius on $W$.  
Put 
$$G_{\ua}(t)= \sum_{i=1}^{s}\psi_p(a_i)+ s\gamma_p-p^{-1} \log(c) +\int_0^t (F_{\ua}(t)-F^{(1)}_{\ua}(t^{\sigma}))\dfrac{dt}{t}, $$
where $\int_0^t(-)\frac{dt}{t}$ is an operator which sends a power series $\sum_{n=1}^{\infty} a_n t^n$ to $\sum_{n=1}^{\infty} \frac{a_n}{n} t^n$. 
Then Asakura \cite{As} defines the \textit{$p$-adic hypergeometric function of logarithmic type} by 
$$\mathscr{F}_{\ua}^{(\sigma)}(t)=\dfrac{G_{\ua}(t)}{F_{\ua}(t)} \in W[[t]]$$
and proves the following congruence relations similar to Dwork's.    
\begin{thm}[{\cite[Theorem 3.2]{As}}]
Suppose that $a_i \not \in \mathbb{Z}_{\leq 0}$ for all $i$. If $p$ is odd, then for all $n \geq 1$, we have
\begin{align*}
\mathscr{F}_{\ua}^{(\sigma)}(t) \equiv \dfrac{[G_{\ua}(t)]_{<p^n}}{[F_{\ua}(t)]_{<p^n}} \pmod{p^n{W}[[t]]}.   
\end{align*}
If $p=2$, then the conguruence above holds modulo $p^{n-1}W[[t]]$.
\end{thm}

We generalize the above function to a function with a general parameter $\ub$ possibly different from $\underline{1}$. 
\begin{dfn}
Let $\ua \in (\Q \cap \Z_p)^s$ (resp. $\ub \in (\Q \cap \Z_p \backslash \Z_{\leq 0})^{s-1})$ be a $s$-tuple (resp. ($s-1$)-tuple) satisfying the conditions (i) and (ii) in Theorem \ref{theorem:1}.  
Put 
$$G_{\ua, \ub}(t)= \sum_{i=1}^{s}\psi_p(a_i)-\sum_{i=1}^{s-1}\psi_p(b_i) +\gamma_p-p^{-1} \log(c) +\int_0^t (F_{\ua, \ub}(t)-F^{(1)}_{\ua, \ub}(t^{\sigma}))\dfrac{dt}{t}. $$
Define the \textit{generalized $p$-adic hypergeometric function of logarithmic type} by 
$$\mathscr{F}_{\ua, \ub}^{(\sigma)}(t)=\dfrac{G_{\ua, \ub}(t)}{F_{\ua, \underline{b}}(t)}.  $$
\end{dfn}

\begin{lem} \label{inW}
We have $G_{\ua, \ub}(t) \in W[[t]]$, therefore
$$\mathscr{F}_{\ua, \underline{b}}^{(\sigma)}(t) \in W[[t]]. $$
\end{lem}

\begin{proof}
Let $G_{\ua, \ub}(t)=\sum D_it^i$, $F_{\ua, \ub}(t)=\sum C_it^i$, and $F^{(1)}_{\ua, \ub}(t)=\sum C^{(1)}_it^i$. 
If $p \nmid i$, then $D_i=C_i/i$ is a $p$-adic integer. 
When $i=mp^k$ with $k \geq 1$ and $(m, p)=1$, 
$D_i$ is written by 
$$D_i=D_{mp^k}=\dfrac{C_{mp^k}-c^{mp^{k-1}}C_{mp^{k-1}}^{(1)}}{mp^k}. $$
Since $c^{mp^{k-1}} \equiv 1 \pmod{p^k}$, it suffices to show that $C_{mp^k} \equiv C_{mp^{k-1}}^{(1)} \pmod{p^k}$. 
It follows from \eqref{ii} below.  
\end{proof}

One of the main theorems in this paper is that $\mathscr{F}^{(\sigma)}_{\ua, \ub}(t)$ satisfies congruence relations similar to Dwork's and Asakura's (see Theorem \ref{main:1}). 
This theorem implies that $\mathscr{F}^{(\sigma)}_{\ua, \ub}(t)$ is a $p$-adic analytic function in the sense of Krasner, i.e. we have the following corollary. 
\begin{cor} \label{maincor}
We have 
$$\mathscr{F}^{(\sigma)}_{\ua, \ub}(t) \in W\langle t, h_{\ua, \ub}(t)^{-1} \rangle:= \varprojlim_{n \geq 1}(W/p^nW[t, h_{\ua, \ub}(t)^{-1}]), \quad h_{\ua, \ub}(t)=\prod_{i=0}^N[F^{(i)}_{\ua, \ub}(t)]_{<p}$$
with some $N \gg 0$. 
\end{cor}
\begin{proof}
Similarly to the proof of \cite[Corollary 2.3]{As}, one can show that 
$$[F_{\ua, \ub}(t)]_{<p^n} \equiv [F_{\ua, \ub}(t)]_{<p}([F_{\ua, \ub}^{(1)}(t)]_{<p})^p \cdots ([F_{\ua, \ub}^{(n-1)}(t)]_{<p})^{p^{n-1}}$$
by Theorem \ref{theorem:1}. 
Note that the set $\{[F_{\ua, \ub}^{(i)}(t)]_{<p} \pmod{p}\}_{i \geq 0}$ of polynomials with $\mathbb{F}_{p}$-coefficients has a finite cardinal. 
Therefore, the assertion follows.  
\end{proof}

\section{Proof of Theorem \ref{main:1}} \label{main_proof}
\subsection{Reduction to the case $c=1$} \label{reduce1}
In the rest of the paper, we fix the following notations. Fix $s \geq 1$ and $\ua=(a_1, \ldots, a_s) \in (\mathbb{Q} \cap \mathbb{Z}_p)^s$, $\ub=(b_1, \ldots, b_{s-1}) \in (\mathbb{Q} \cap \mathbb{Z}_p\backslash \mathbb{Z}_{\leq 0})^{s-1}$ satisfying the conditions (i) and (ii) in Theorem \ref{theorem:1}.    
For $c \in1+pW$, let $\sigma(t)=ct^p$ be a Frobenius on $W[[t]]$.  
Put 
$$C_n:=\dfrac{(a_1)_n(a_2)_n \cdots (a _s)_n}{ (b_1)_n \cdots (b_{s-1})_n(1)_n}, \quad C^{(1)}_n:=\dfrac{(a'_1)_n(a'_2)_n \cdots (a' _s)_n}{(b'_1)_n \cdots (b'_{s-1})_n(1)_n}$$
for $n \geq 0$. 
Define $D_n$ by $G_{\ua, \ub}(t)=\sum_{n=0}^{\infty}D_n t^n$, or explicitly 
\begin{align} \label{note}
  \begin{split}
&D_0=\sum_{i=1}^{s}\psi_p(a_i)-\sum_{i=1}^{s-1}\psi_p(b_i) +\gamma_p-p^{-1} \log(c), \\
&D_n = \dfrac{C_n}{n} \ (p \nmid n), \quad D_{mp^k}=\dfrac{C_{mp^k}-c^{mp^{k-1}}C^{(1)}_{mp^{k-1}}}{mp^k} \ (m, k \geq 1). 
\end{split}
\end{align}

\begin{lem}\label{lemma:4} \label{reduce}
The proof of Theorem \ref{main:1} is reduced to the case $c=1$. 
\end{lem}
\begin{proof}
Let $D=t\frac{d}{dt}$. Then we can prove that 
$$\dfrac{D^k F_{\ua, \ub}(t)}{F_{\ua, \ub}(t)} \equiv \dfrac{[D^kF_{\ua, \ub}(t)]_{< p^n}}{F_{\ua, \ub}(t)_{<p^n}} \pmod{p^n}$$
by using Theorem \ref{theorem:1} (see the proof of \cite[Theorem 2.4]{As2}). 
Using this result, we can prove the lemma similarly to the proof of \cite[Lemma 3.5]{As}. 
\end{proof}

\subsection{Preliminary lemmas} \label{Prelem} 
In this section, we prepare some lemmas to prove the congruence relations according to \cite[Section 3.3]{As}. 
In the rest of the section, we suppose $c=1$. 
Let  
\begin{align*}
D_0=\sum_{i=1}^{s}\psi_p(a_i)-\sum_{i=1}^{s-1}\psi_p(b_i) +\gamma_p, \hspace{5mm} D_i=\dfrac{C_i-C^{(1)}_{i/p}}{i} \quad  (i\in \mathbb{Z}_{\geq 1}), 
\end{align*}
where the notations are as in \eqref{note} and we define $C_{i/p}=C^{(1)}_{i/p}=0$ unless $p \mid i$.

\begin{lem}\label{lemma:2}
For an $p$-adic integer $\alpha \in \mathbb{Z}_p$ and $n \in \mathbb{Z}_{\geq 1}$, we define
$$\{\alpha\}_n=\prod_{\substack {1\leq i \leq n \\ p \nmid (\alpha +i-1)}}(\alpha+i-1),$$
and $\{\alpha\}_0=1$. 
Let $a, b \in \Z_p \backslash \Z_{\leq 0}$ such that $[a]_0< [b]_0$. 
Then for any $m \in \mathbb{Z}_{\geq 0}$, we have the following. 
\begin{enumerate}
\item If $m \equiv 0, 1, \ldots, [a]_0 \pmod{p}$, then we have
$$\dfrac{(a)_m}{(b)_m} \left(\dfrac{(a^{\prime})_{\lfloor m/p \rfloor}}{(b^{\prime})_{\lfloor m/p \rfloor}}\right)^{-1}=\dfrac{\{a\}_m}{\{b\}_m}. $$
\item If $m \equiv [a]_0+1, \ldots, [b]_0 \pmod{p}$, then we have
$$\dfrac{(a)_m}{(b)_m} \left(\dfrac{(a^{\prime})_{\lfloor m/p \rfloor}}{(b^{\prime})_{\lfloor m/p \rfloor}}\right)^{-1}=\left(a+[a]_0+p\lfloor \dfrac{m}{p} \rfloor\right)\dfrac{\{a\}_m}{\{b\}_m}. $$
\item If $m \equiv [b]_0+1, \ldots, p-1 \pmod{p}$, then we have 
$$\dfrac{(a)_m}{(b)_m} \left(\dfrac{(a^{\prime})_{\lfloor m/p \rfloor}}{(b^{\prime})_{\lfloor m/p \rfloor}}\right)^{-1}=\dfrac{\left(a+[a]_0+p\lfloor \dfrac{m}{p} \rfloor\right)}{\left(b+[b]_0 +p\lfloor \dfrac{m}{p} \rfloor\right)}\dfrac{\{a\}_m}{\{b\}_m}. $$
\end{enumerate}
\end{lem}

\begin{proof}
This follows from \cite[Lemma 3.6]{As}. 
\end{proof}

\begin{cor}[Dwork] \label{lemma:17}
For any $m_1, m_2 \in \mathbb{Z}_{ \geq 0}$, if $m_1 \equiv m_2  \pmod{p^n}$, then we have 
$${C_{m_1}}{C^{(1)}_{\lfloor m_2/p \rfloor}} \equiv  {C_{m_2}} {C^{(1)}_{\lfloor m_1/p \rfloor}} \pmod{p^n}. $$
\end{cor}
\begin{rmk}
If $\ub \neq \underline{1}$, then $C_m/C^{(1)}_{\lfloor m/p \rfloor}$ is not a $p$-adic integer in general.   
\end{rmk}
\begin{proof}
This follows from Theorem \ref{theorem:1}, or 
we can easily show this by using  
Lemma \ref{lemma:2} on noticing the fact that 
for any $\alpha \in \mathbb{Z}_p$, we have 
\begin{align*}
&\left(a_i+[a_i]_0+p\lfloor \dfrac{m_1}{p} \rfloor\right) \equiv \left(a_i+[a_i]_0+p\lfloor \dfrac{m_2}{p} \rfloor\right), \\
& \{\alpha \}_{p^n} \equiv \prod_{i \in (\Z/p^n\Z)^*} i \equiv 
\left\{
\begin{array}{ll}
1 & (p=2, n \neq 2), \\
-1 & (otherwise)
\end{array}
\right.
\end{align*}
modulo $p^n$. 
\end{proof}

\begin{lem}\label{lemma:1}
Let $a, b \in \mathbb{Z}_p \backslash \mathbb{Z}_{\leq 0}$ and $m, n \in \Z_{\geq 1}$. 
Then we have 
\begin{align} \label{i}
1-\dfrac{(a^{\prime})_{mp^{n-1}}}{(b^{\prime})_{mp^{n-1}}}\left( \dfrac{(a)_{ mp^n}}{(b)_{mp^n}}\right)^{-1} \equiv mp^n(\psi_p(a) -\psi_p(b)) \pmod{p^{2n}}. 
\end{align}
Moreover, 
$C_{mp^{n-1}}^{(1)}/C_{mp^n}$ and $D_k /C_k$ are $p$-adic intgers for all $k, m \geq 0$, $n \geq 1$, and 
\begin{align}
&\dfrac{C_{mp^{n-1}}^{(1)}}{C_{mp^n}} \equiv 1-mp^n\left(\sum_{i=1}^s\psi_p(a_i) -\sum_{i=1}^{s-1}\psi_p(b_i) + \gamma_p \right)\pmod{p^{2n}}, \label{ii} \\
&p \nmid m \quad \Longrightarrow \quad  \dfrac{D_{mp^n}}{C_{mp^n}}=\dfrac{1-C^{(1)}_{mp^{n-1}}/C_{mp^n}}{mp^n} \equiv D_0 \pmod{p^n}. \label{iii}
\end{align}
\end{lem}

\begin{proof}
By Lemma \ref{lemma:2} (i), we have $C_{mp^{n-1}}^{(1)}/{C_{mp^n}} \in \Z_p$. 
It is enough to show \eqref{i} since \eqref{ii} follows from \eqref{i} and \eqref{iii} follows from \eqref{ii}. 
Moreover, \eqref{iii} implies that $D_k/C_k \in \Z_p$ for any $k \geq 0$.   

Let us show \eqref{i}. By \cite[Lemma 3.8]{As}, we have 
$$1-\dfrac{(a^{\prime})_{mp^{n-1}}}{(1)_{mp^{n-1}}}\left(\dfrac{(a)_{mp^n}}{(1)_{mp^{n}}}\right)^{-1} \equiv mp^n(\psi_p(a)+\gamma_p) \pmod{p^{2n}}$$
for any $a \in \Z_p \backslash \Z_{ \leq 0}$ and $m, n \in \Z_{\geq 1}$. 
Hence we have 
\begin{align*}
\dfrac{(a^{\prime})_{mp^{n-1}}}{(b^{\prime})_{mp^{n-1}}}\left( \dfrac{(a)_{ mp^n}}{(b)_{mp^n}}\right)^{-1} 
 &\equiv \dfrac{(a^{\prime})_{mp^{n-1}}}{(1)_{mp^{n-1}}}\left(\dfrac{(a)_{mp^n}}{(1)_{mp^{n}}}\right)^{-1}\left[ \dfrac{(b^{\prime})_{mp^{n-1}}}{(1)_{mp^{n-1}}}\left(\dfrac{(b)_{mp^n}}{(1)_{mp^{n}}}\right)^{-1}\right]^{-1}\\
&\equiv (1-mp^n(\psi_p(a)+\gamma_p))(1-mp^n(\psi_p(b)+\gamma_p))^{-1}\\
&= (1-mp^n(\psi_p(a)+\gamma_p))\sum_{i=0}^{\infty}(mp^n(\psi_p(b)+\gamma_p))^i\\
&\equiv (1-mp^n(\psi_p(a)+\gamma_p))(1+mp^n(\psi_p(b)+\gamma_p))\\
&\equiv 1-mp^n(\psi_p(a)-\psi_p(b))
\end{align*}
modulo $p^{2n}$, which completes the proof of \eqref{i}. 
\end{proof}

\begin{lem} \label{lemma:19}
For any $m_1, m_2 \in \mathbb{Z}_{\geq 0}$ and $n \in \Z_{\geq 1}$, 
if $m_1 \equiv m_2 \pmod{p^n}$, then we have
$$\dfrac{D_{m_1}}{C_{m_1}} \equiv \dfrac{D_{m_2}}{C_{m_2}} \pmod{p^n}.$$
\end{lem}

\begin{proof}
If $p \nmid m$, then $D_m/C_m=1/m$ and hence the assertion is obvious. Let $m_1 =kp^i$ with $i \geq 1$ and $p \nmid k$. It is enough to show the assertion in case $m_2=m_1+p^n$. If $n \leq i$, then 
$$\dfrac{D_{m_1}}{C_{m_1}} \equiv \dfrac{D_{m_2}}{C_{m_2}} \equiv D_0 \pmod{p^n} $$
by \eqref{iii}. 
Suppose $n > i$. Note that 
$$1-m\dfrac{D_m}{C_m} =\dfrac{C_{m/p}^{(1)}}{C_m}=\prod_{r=1}^{s}\dfrac{\{b_r\}_m}{\{a_r\}_m} $$
by Lemma \ref{lemma:2} (i). 
Here, we put $b_s:=1$. 
We have
\begin{align*}
1-m_2\dfrac{D_{m_2}}{C_{m_2}}   
&= \prod_r\dfrac{\{b_r\}_{kp^i+p^n}}{\{a_r\}_{kp^i+p^n}}
= \prod_r\dfrac{\{b_r\}_{kp^i}}{\{a_r\}_{kp^i}}\dfrac{\{b_r+kp^i\}_{p^n}}{\{a_r+kp^i\}_{p^n}}\\
&\equiv \left(1-m_1\dfrac{D_{m_1}}{C_{m_1}} \right) \prod_r \dfrac{\{b_r+kp^i\}_{p^n}}{\{a_r+kp^i\}_{p^n}}  \\
&\equiv \left(1-m_1\dfrac{D_{m_1}}{C_{m_1}} \right) \prod_r(1-p^n(\psi_p(a_r+kp^i)-\psi_p(b_r+kp^i))) \pmod{p^{2n}}. 
\end{align*}
The last congruence follows from Lemma \ref{lemma:2} (i) and \eqref{i}. 
For $(p, i) \neq (2, 1)$, the congruence
\begin{align}
\psi_p(a_r+kp^i)-\psi_p(b_r+kp^i) \equiv \psi_p(a_r)-\psi_p(b_r) \pmod{p^i} \label{eq:7}
\end{align}
follows from (\ref{eq:6}).  
On the other hand, for $z \in \Z_2$, we have  
$$\psi_2(z+2)-\psi_2(z) \equiv 1 \pmod{2}$$
(cf. \cite[Theorem 2.6 (3)]{As}), 
hence the congruence (\ref{eq:7}) is also correct in the case $(p, i)=(2, 1)$. 
It concludes that 
\begin{align*}
1-m_2\dfrac{D_{m_2}}{C_{m_2}} \equiv \left(1-m_1\dfrac{D_{m_1}}{C_{m_1}} \right)(1-p^nD_0) \pmod{p^{n+i}}. 
\end{align*}
Therefore, we have
 \begin{align}
kp^i\left(\dfrac{D_{m_2}}{C_{m_2}}-\dfrac{D_{m_1}}{C_{m_1}}\right) \equiv -p^n\dfrac{D_{m_2}}{C_{m_2}}+p^nD_0=p^n\left(D_0-\dfrac{D_{m_2}}{C_{m_2}}\right)\pmod{p^{n+i}}  \label{eq:12}
 \end{align}
We note that $m_2=kp^i+p^n=p^i(k+p^{n-i})$. 
By \eqref{iii}, we have 
$$\dfrac{D_{m_2}}{C_{m_2}} \equiv D_0 \pmod{p^i}$$
hence the right-hand side of (\ref{eq:12}) vanishes. This proves the lemma.  
\end{proof}

\begin{lem} \label{lm}
Put $S_m=\sum_{i+j=m}(C_{i+p^n}D_{j}-C_{i}D_{j+p^n})$ for $m \in \mathbb{Z}_{\geq 0}$. 
Then we have 
$$S_m \equiv \sum_{i+j=m}(C_{i+p^n}C_{j}-C_iC_{j+p^n})\dfrac{D_j}{C_j} \pmod{p^n}.$$
\end{lem}
\begin{proof}
By Lemma \ref{lemma:19}, we have
\begin{align*}
S_m &= \sum_{i+j=m}\left(C_{i+p^n}D_{j}-C_{i} C_{j+p^n} \dfrac{D_{j+p^n}}{C_{j+p^n}}\right)\\
&\equiv \sum_{i+j=m} \left(C_{i+p^n}D_{j}-C_{i} C_{j+p^n} \dfrac{D_{j}}{C_{j}}\right)  \pmod{p^n} \\
&=  \sum_{i+j=m}(C_{i+p^n}C_{j}-C_iC_{j+p^n})\dfrac{D_j}{C_j}, 
\end{align*}
which finishes the proof. 
\end{proof}

The following lemma is a slight modification of \cite[Lemma 3.12]{As}. 

\begin{lem}  \label{keylem}
For all $m, k, s \in \mathbb{Z}_{\geq 0}$ and $0 \leq l \leq n$, we have
\begin{align} \label{key}
\sum_{\substack{i+j=m\\ i \equiv k \ {\rm mod} \ {p^{n-l}}}} \left(C_iC_{j+p^{n}}-C_jC_{i+p^{n}} \right) \equiv 0 \pmod{p^{l+1}}.
\end{align}
\end{lem}
\begin{proof}
When $l=n$, the lemma is obvious since the left-hand side of \eqref{key} is 
$$\sum_{i+j=m} \left(C_iC_{j+p^{n}}-C_jC_{i+p^{n}} \right)=0. $$ 
Suppose that $0 \leq l \leq n-1$. 
For $k \in \Z_{\geq 0}$, we put
$$F_k(t)= \sum_{i \equiv k \ {\rm mod}\ {p^{n-l}}}C_i t^i. $$
Then \eqref{key} is equivalent to 
\begin{align} \label{key2}
F_k(t) \cdot [F_{m-k}(t)]_{<p^n} \equiv [F_k(t)]_{<p^n} \cdot F_{m-k}(t) \pmod{p^{l+1}}. 
\end{align} 
Since we have the congruence 
$$\dfrac{F^{(i)}_{\ua, \ub}(t)}{F^{(i+1)}_{\ua, \ub}(t^p)} \equiv \dfrac{[F^{(i)}_{\ua, \ub}(t)]_{<p^n}}{[F^{(i+1)}_{\ua, \ub}(t^p)]_{<p^n}} \pmod{p^n}$$
by Theorem \ref{theorem:1}, 
we can prove \eqref{key2} similarly as in the proof of \cite[Lemma 3.12]{As}, considering the case $(d, j)=(n, m-k)$ in loc. cit.  
\end{proof}

\subsection{Proof of Theorem \ref{main:1}}
We finish the proof of Theorem \ref{main:1}. 
We note that the statement of Theorem \ref{main:1} is equivalent to 
$$S_m \equiv 0 \pmod{p^n}$$
for all $m \geq 0$. 
Put $q_k=D_k/C_k$. 
By Lemma \ref{lm} and Lemma \ref{lemma:19}, we have
$$S_m \equiv \sum_{k=0}^{p^n-1}q_k \overbrace{\sum_{\substack{i+j=m\\j\equiv k \ {\rm mod} \ p^n}} (C_{i+p^n}C_j-C_iC_{j+p^n})}^{(*)} \pmod{p^n}.$$
It follows from Lemma \ref{keylem} that  $(*)$ is zero modulo ${p}$. Therefore, again by Lemma \ref{lemma:19}, one can rewrite  
$$S_m \equiv \sum_{k=0}^{p^{n-1}-1}q_k \overbrace{\sum_{\substack{i+j=m\\j\equiv k \ {\rm mod} \ p^{n-1}}} (C_{i+p^n}C_j-C_iC_{j+p^n})}^{(**)} \pmod{p^n}.$$
It follows from Lemma \ref{keylem} that  $(**)$ is zero modulo ${p^2}$. Therefore, again by Lemma \ref{lemma:19}, one can rewrite  
 $$S_m \equiv \sum_{k=0}^{p^{n-2}-1}q_k \sum_{\substack{i+j=m\\j\equiv k \ {\rm mod} \ p^{n-2}}} (C_{i+p^n}C_j-C_iC_{j+p^n}) \pmod{p^n}.$$
 Continuing the same discussion, one finally obtains 
$$S_m \equiv \sum_{i+j=m} (C_{i+p^n}C_j-C_iC_{j+p^n}) =0 \pmod{p^n}, $$
which finishes the proof.

\section{Special values at $t=1$} \label{special_value}
In this section, we give numerical computations of the special value of $\mathscr{F}^{(\sigma)}_{a_1, a_2; b}(t)$ at $t=1$ 
and prove Theorem \ref{main:3}. 
In Section \ref{proof2}, we prove that we can define the special value at $t=1$ under some assumptions, and give numerical computations modulo $p^4$.  
In Section \ref{HG}, we recall some properties of hypergeometric curves and the relation between $p$-adic regulators of them and $p$-adic hypergeometric functions of logarithmic type, which is proved by Asakura \cite{As}. 
In Section \ref{preg}, we compute p-adic regulators of hypergeometric curves, which differ from Asakura's, as we use a different Frobenius.  
In Section \ref{proof3}, we prove Theorem \ref{main:3} by comparing $p$-adic regulators of a hypergeometric curve.

\subsection{Numerical computations} \label{proof2}
\begin{thm} \label{main:2}
Let $N$ be a positive integer and $p$ be a prime such that $N \mid p-1$. 
Let $i, j, k \in \{1, \ldots, N\}$ be integers such that $i+j \leq k$.
Then we can define the special value $\mathscr{F}_{\frac{i}{N}, \frac{j}{N};\frac{k}{N}}^{(\sigma)}(1)$ 
and we have some numerical computations $\mathscr{F}_{\frac{i}{N}, \frac{j}{N};\frac{k}{N}}^{(\sigma)}(1) \pmod{p^4}$ for $N=2, 3, 4, 5, 6$ and $p=3, 5, 7, 11, 13$ as Table \ref{numerical}.  
\begin{table}
 \caption{The numerical computations of $\mathscr{F}_{\frac{i}{N}, \frac{j}{N};\frac{k}{N}}^{(\sigma)}(1) \pmod{p^4}$}  \label{numerical}
\scalebox{1}[1]
{
\begin{tabular}{cccccc} 
\hline
$\overset{}(p, N, i, j, k)$ & $\mathscr{F}_{\frac{i}{N}, \frac{j}{N};\frac{k}{N}}^{(\sigma)}(1) \pmod{p^4}$ &  $\underset{}(p, N, i, j, k)$ & $\mathscr{F}_{\frac{i}{N}, \frac{j}{N};\frac{k}{N}}^{(\sigma)}(1)\pmod{p^4}$  \\ \toprule
$(\overset{}{3, 2, 1, 1, 2})$ & $0$ &  $(11, 5, 1, 2, 4)$ & $12680$  \\
$(\overset{}{5, 2, 1, 1, 2})$ & $0$ &  $(11, 5, 1, 2, 5)$ & $2926$\\
$(\overset{}{5, 4, 1, 1, 2})$ & $0$ &   $(11, 5, 1, 3, 4)$ & $0$ \\
$(\overset{}{5, 4, 1, 1, 3})$ & $131$ & $(11, 5, 1, 3, 5)$ & $180$ \\
$(\overset{}{5, 4, 1, 1, 4})$ & $94$  &  $(11, 5, 1, 4, 5)$  & $0$  \\
$(\overset{}{5, 4, 1, 2, 3})$ & $0$  &  $(11, 5, 2, 2, 4)$ & $0$  \\
$(\overset{}{5, 4, 1, 2, 4})$ & $604$  & $(11, 5, 2, 2, 5)$& $10991$   \\
$(\overset{}{5, 4, 1, 3, 4})$ & $0$  &   $(11, 5, 2, 3, 5)$& $0$\\
$(\overset{}{7, 2, 1, 1, 2})$ & $0$  & $(13, 2, 1, 1, 2)$ & $0$ \\
$(\overset{}{7, 3, 1, 1, 2})$ & $0$  &   $(13, 3, 1, 1, 2)$ & $0$ \\
$(\overset{}{7, 3, 1, 1, 3})$ & $290$  &   $(13, 3, 1, 1, 3)$ & $18112$\\
$(\overset{}{7, 3, 1, 2, 3})$ & $0$  &  $(13, 3, 1, 2, 3)$ & $0$ \\
$(\overset{}{7, 6, 1, 1, 2})$ &  $0$  &   $(13, 4, 1, 1, 2)$ & $0$ \\
$(\overset{}{7, 6, 1, 1, 3})$ &  $985$  &  $(13, 4, 1, 1, 3)$ & $24856$ \\
$(\overset{}{7, 6, 1, 1, 4})$ &  $831$  &  $(13, 4, 1, 1, 4)$ & $19301$ \\
$(\overset{}{7, 6, 1, 1, 5})$ &  $1058$  &  $(13, 4, 1, 2, 3)$ & $0$ \\
$(\overset{}{7, 6, 1, 1, 6})$ &  $481$  &  $(13, 4, 1, 2, 4)$ & $1084$ \\
$(\overset{}{7, 6, 1, 2, 3})$ &  $0$  &  $(13, 4, 1, 3, 4)$ & $0$\\
$(\overset{}{7, 6, 1, 2, 4})$ &  $1926$  &  $(13, 6, 1, 1, 2)$ & $0$ \\
$(\overset{}{7, 6, 1, 2, 5})$ &  $1571$  &  $(13, 6, 1, 1, 3)$ & $13217$ \\
$(\overset{}{7, 6, 1, 2, 6})$ &  $1678$  &  $(13, 6, 1, 1, 4)$ & $11029$ \\
$(\overset{}{7, 6, 1, 3, 4})$ &  $0$  & $(13, 6, 1, 1, 5)$ & $1195$ \\
$(\overset{}{7, 6, 1, 3, 5})$ &  $1616$  &  $(13, 6, 1, 1, 6)$ & $14792$ \\
$(\overset{}{7, 6, 1, 3, 6})$ &  $1869$  & $(13, 6, 1, 2, 3)$ & $0$\\
$(\overset{}{7, 6, 1, 4, 5})$ &  $0$  & $(13, 6, 1, 2, 4)$ & $21091$ \\
$(\overset{}{7, 6, 1, 4, 6})$ &  $324$  &  $(13, 6, 1, 2, 5)$ & $7884$ \\
$(\overset{}{7, 6, 1, 5, 6})$ &  $0$  & $(13, 6, 1, 2, 6)$ & $7433$ \\
$(\overset{}{7, 6, 2, 3, 5})$ &  $0$  & $(13, 6, 1, 3, 4)$ & $0$ \\
$(\overset{}{7, 6, 2, 3, 6})$ &  $2160$  & $(13, 6, 1, 3, 5)$ & $19795$ \\
$\overset{} (11, 2, 1, 1, 2)$ &  $0$  & $(13, 6, 1, 4, 5)$ & $0$  \\
$\overset{} (11, 5, 1, 1, 2)$ & $0$  & $(13, 6, 1, 4, 6)$ & $20137$ \\
$\overset{} (11, 5, 1, 1, 3)$  & $4469$  & $(13, 6, 1, 5, 6)$ & $0$ \\
$\overset{} (11, 5, 1, 1, 4)$  & $2709$  & $(13, 6, 2, 3, 5)$ & $0$   \\
$\overset{} (11, 5, 1, 1, 5)$   & $3590$  & $(13, 6, 2, 3, 6)$ & $11998$  \\
$\overset{} (11, 5, 1, 2, 3)$  & $0$  &  \\
 \end{tabular}
 }
 \end{table}
\end{thm}
\begin{rmk}
Some values for $(p, N, 1, N-1, N)$ in Table \ref{numerical} have already been computed by Kayaba in his unpublished master's thesis \cite{Kayaba}.  
\end{rmk}

\begin{proof}[Proof of Theorem \ref{main:2}]
By the assumption $i+j \leq k$, 
$\ua=\left(\frac{i}{N}, \frac{j}{N} \right)$ and $\ub=\left( \frac{k}{N} \right)$ satisfy the conditions (i) and (ii) in Theorem \ref{theorem:1}. 
Let us show $|h_{\ua, \ub}(1)|_p=1$. 
By Theorem \ref{theorem:1} and the assumption $N \mid p-1$, it suffices to show that 
$$[F_{\ua, \ub}(t)]_{<p}|_{t=1} \not \equiv 0 \pmod{p}. $$ 
Let $i_0, j_0, k_0 \in \{0, 1, \ldots, p-1\}$ be the integers such that $i/N \equiv -i_0$, $j/N \equiv -j_0$ and $k/N \equiv -k_0 \pmod{p}$.  
Then we have 
$$[F_{\ua, \ub}(t)]_{<p}=\sum_{n=0}^{p-1}\dfrac{(\frac{i}N)_n  (\frac{j}N)_n}{(\frac{k}N)_n  (1)_n}t^n \equiv \sum_{n=0}^{p-1}\dfrac{(-i_0)_n  (-j_0)_n}{(p-k_0)_n  (1)_n}t^n \pmod{p}. $$
Since $-i_0$ and $-j_0$ are non-positive integers greater than $-p$, we have 
$$\sum_{n=0}^{p-1}\dfrac{(-i_0)_n  (-j_0)_n}{(p-k_0)_n  (1)_n}t^n =\sum_{n=0}^{\infty}\dfrac{(-i_0)_n  (-j_0)_n}{(p-k_0)_n  (1)_n}t^n={_{2}F_{1}}\left( 
\begin{matrix}
-i_0, -j_0 \\
p-k_0
\end{matrix}
; t
\right). $$ 
By Gauss's formula (cf \cite[1.1.5]{Slater}), we have 
\begin{align*}
{_{2}F_{1}}\left( 
\begin{matrix}
-i_0, -j_0 \\
p-k_0
\end{matrix}
; 1
\right)
&=
\dfrac{\Gamma(p-k_0) \Gamma(p-k_0 + i_0 + j_0)}{\Gamma(p-k_0 + i_0) \Gamma(p-k_0+j_0)} \\
&= \dfrac{(p-k_0-1)! (p-k_0+i_0+j_0-1)!}{(p-k_0+i_0-1)! (p-k_0+j_0-1)!} \\
& \not \equiv 0 \pmod{p}, 
\end{align*}
hence $[F_{\ua, \ub}(t)]_{<p}|_{t=1} \not \equiv 0 \pmod{p}$.  
Therefore, we can define the special value 
\begin{align*} 
\mathscr{F}_{\frac{i}{N}, \frac{j}{N};\frac{k}{N}}^{(\sigma)}(1)=\lim_{n \to \infty} \left(\left. \dfrac{[G_{\frac{i}{N}, \frac{j}{N}; \frac{k}{N}}(t)]_{<p^n}}{[F_{\frac{i}{N}, \frac{j}{N}; \frac{k}{N}}(t)]_{<p^n}} \right|_{t=1} \right). 
\end{align*}
The command of Mathematica for $\mathscr{F}_{a, b; c}^{(\sigma)}(1) \pmod{p^4}$ is as follows: 
\begin{figure}[H]
\centering
\includegraphics[width=13cm]{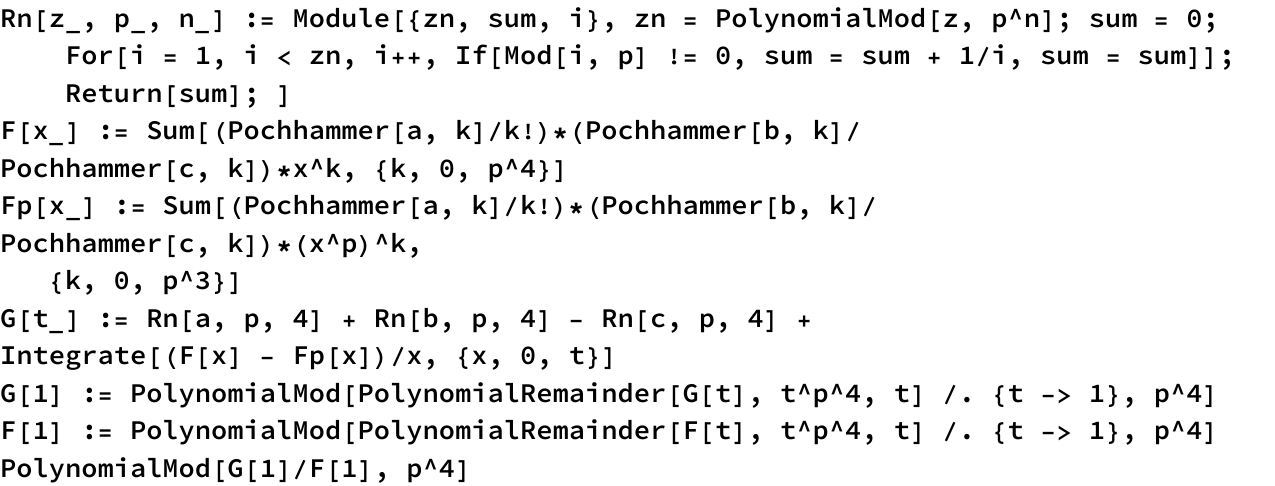}
\end{figure}
\end{proof}

\subsection{Properties of hypergeometric curves} \label{HG}
For a commutative ring $A$, we denote the projective line over $A$ with homogeneous coordinate $(Z_0, Z_1)$ by $\mathbb{P}^1_A(Z_0, Z_1)$.
Let $W=W(\overline{\mathbb{F}}_p)$ be the Witt ring of $\F$ and $T=W[t, (t-t^2)^{-1}]$.  
Let $N \geq 2$ be an integer and $p$ a prime greater than $N$.  
Let $X$ be a projective scheme over $T$ whose affine equation is 
$$(1-x^N)(1-y^N)=t. $$
This is called a \textit{hypergeometric curve} over $T$. 
\begin{lem}[{\cite[Lemma 4.1]{As}, \cite[Section 2.7]{AO}}] \label{proj}
The morphism $X \to \Spec T$ extends to a projective flat morphism 
\begin{align*} 
f \colon Y \to \mathbb{P}^1_W=\mathbb{P}^1_W(T_0, T_1), \quad t=T_0/T_1
\end{align*}
of a smooth projective $W$-scheme 
satisfying the following conditions. 
\begin{enumerate}
\item 
The fiber $f^{-1}(t=0)$ is a union of $2N$ rational curves 
$$x=\nu_1, \quad y = \nu_2, \quad (\nu_1, \nu_2 \in \mu_N), $$
intersecting each other transversally at $(\nu_1, \nu_2)$. 
\item The fiber $f^{-1}(t=1)$ is a union of the Fermat curve of degree $N$ and a rational curve with multiplicity $N$, intersecting each other transversally at $N$ points. 
\item The fiber $f^{-1}(t=\infty)$ is a union of two rational curves both with multiplicity $N$, intersecting each other transversally at one point. 
\end{enumerate}
\end{lem}

Let $K= \operatorname{Frac}{W}$ be the fractional field of $W$. 
Let $S=\Spec T\subset \mathbb{P}^1_W$.  
For a $W$-scheme $Z$, we write $Z_{K}=K \times_W Z$.   
The group $\mu_N \times \mu_N$ acts on $Y$ in the following way: 
$$[\zeta, \nu] \cdot (x, y, t):=(\zeta x, \nu y, t), \quad  (\zeta, \nu) \in \mu_N \times \mu_N. $$
For a $K$-vector space $V$ with an action of $\mu_N \times \mu_N$, let $V^{(i, j)}$ be the vector subspace on which $(\zeta, \mu)$ acts by multiplication $\zeta^i \nu^j$ for all $(\zeta, \nu) \in \mu_N \times \mu_N$. 
Then we have the eigendecomposition 
$$H_{\rm dR}^1(X_K / S_K) = \bigoplus_{i, j} H^1_{\rm dR} (X_K /S_K)^{(i, j)}, $$
where each eigenspace $H^1_{\rm dR} (X_K /S_K)^{(i, j)}$ is free of rank $2$ over $\mathscr{O}({S_K})$ (see \cite[Lemma 2.2]{As4}).   
 Put  
 \begin{align*}
 &a_i:= 1-\dfrac{i}{N}, \quad b_j:=1-\dfrac{j}{N}, \\
 &\omega_{i, j}:=N \dfrac{x^{i-1}y^{j-N}}{1-x^N}dx = -N \dfrac{x^{i-N}y^{j-1}}{1-y^N}dy, \\
 &\eta_{i, j}:=\dfrac{1}{x^N-1+t} \omega_{i, j} =Nt^{-1} x^{i-N} y^{j-N-1} dy
 \end{align*}
for integers $i$, $j$ such that $1 \leq i, j \leq N-1$. 
Then $\{\omega_{i, j}, \eta_{i, j}\}$ form an $\mathscr{O}(S_K)$-free basis of $H^1_{\rm dR} (X_K /S_K)^{(i, j)}$. 
Let 
$$F_{a_i, b_j;1}(t)={_{2}F_{1}}\left( 
\begin{matrix}
a_i, b_j \\
1
\end{matrix}
; t
\right)
=
\sum_{n=0}^{\infty}\dfrac{(a_i)_n  (b_j)_n}{(1)_n  (1)_n}t^n \in K[[t]]
$$
be a hypergeometric series. 
Put 
\begin{align*}
\tw_{i, j}:=\dfrac1{F_{a_i, b_j;1}(t)} \omega_{i, j}, \quad \widetilde{\eta}_{i, j}:=-t (1- t)^{a_i+b_j} (F'_{a_i, b_j;1}(t) \omega_{i, j} + b_j F_{a_i, b_j;1}(t) \eta_{i, j}). 
\end{align*}

Let $\sigma$ be the $p$th Frobenius on $W[t, (t-t^2)^{-1}]^{\dagger}$, the ring of overconvergent power series defined by $\sigma(t)=t^p$, which extends on $\mathscr{O}(S_K)^{\dagger}=K[t, (t-t^2)^{-1}]^{\dagger}:=K \otimes_W W[t, (t-t^2)^{-1}]^{\dagger}$. 
For an integer $r$, let $\mathscr{O}(S_K)^{\dagger}(r)$ be the Tate twist of $\mathscr{O}(S_K)^{\dagger}$, i.e. 
$\mathscr{O}(S_K)^{\dagger}(r)=\mathscr{O}(S_K)^{\dagger}$ as an $\mathscr{O}(S_K)$-module, and the $p$th Frobenius $\sigma$ acts on $\mathscr{O}(S_K)^{\dagger}(r)$ by $p^{-r} \sigma$. 

We consider the Milnor symbol 
\begin{align} \label{xi}
\xi= \xi(\nu_1, \nu_2):= \left \{\dfrac{x-1}{x-\nu_1}, \dfrac{y-1}{y-\nu_2} \right\} \in K_2^M(\mathscr{O}(X)), \quad \nu_1, \nu_2 \in \mu_N \backslash \{1\}. 
\end{align}
By \cite[Section 2.6]{AM}, we have a following exact sequence 
$$0 \to \mathscr{O}(S_K)^{\dagger}(2) \otimes_{\mathscr{O}(S_K)} H^1_{\rm dR} (X_K/S_K) \to \mathscr{O}(S_K)^{\dagger} \otimes_{\mathscr{O}(S_K)} M_{\xi}(X_K/S_K)  \to \mathscr{O}(S_K)^{\dagger} \to 0$$
endowed with 
\begin{itemize}
\item Frobenius $\Phi_{\sigma}$-action which is $\sigma$-linear, 
\item $\operatorname{Fil}^i \subset M_{\xi}(X_K/S_K) $ (Hodge filtration) with
$$\mathscr{O}(S_K)^{\dagger} \otimes_{\mathscr{O}(S_K)} \operatorname{Fil}^0M_{\xi}(X_K/S_K)  \xrightarrow{\simeq} \mathscr{O}(S_K)^{\dagger}. $$
\end{itemize}
In particular, there exists a unique lifting $e_{\xi}$ in $\mathscr{O}(S_K)^{\dagger} \otimes_{\mathscr{O}(S_K)} \operatorname{Fil}^0M_{\xi}(X_K/S_K)$ of $1 \in \mathscr{O}(S_K)^{\dagger}$. 
Then the element $e_{\xi} - \Phi_{\sigma} (e_{\xi})$ defines an element of $\mathscr{O}(S_K)^{\dagger}(2) \otimes_{\mathscr{O}(S_K)} H^1_{\rm dR} (X_K/S_K)$, which corresponds to the image of $\xi$ under the $p$-adic regulator (\cite[Proposition 4.3]{AM}).

Define $\e_{k, \sigma}^{(i, j)}(t)$ and $E_{k, \sigma}^{(i, j)}(t)$ for $k=1, 2$ by 
 \begin{align*}
 e_{\xi} -\Phi_{\sigma}(e_{\xi})
 &=\dfrac1{N^2} \sum_{i, j=1}^N (1 - \nu_1^{-i})(1- \nu_2^{-j}) [\e_{1, \sigma}^{(i, j)}(t) \omega_{i, j}+\e_{2, \sigma}^{(i, j)}(t) \eta_{i, j}] \\
 &=\dfrac1{N^2} \sum_{i, j=1}^N (1 - \nu_1^{-i})(1- \nu_2^{-j}) [E_{1, \sigma}^{(i, j)}(t) \tw_{i, j}+E_{2, \sigma}^{(i, j)}(t) \widetilde{\eta}_{i, j}]. 
 \end{align*}

\begin{thm}[{\cite[Theorem 4.8]{As}}]
We have 
$$\dfrac{E_{1, \sigma}^{(i, j)}(t)}{F_{a_i, b_j;1}(t)} = -\mathscr{F}^{(\sigma)}_{a_i, b_j; 1}(t). $$
\end{thm}

\subsection{$p$-adic regulators of hypergeometric curves} \label{preg}
Put $\lm =1-t$ and let $\tau$ be another $p$th Frobenius on $W[\lm, (\lm-\lm^2)^{-1}]^{\dagger}$ defined by $\tau(\lm)=\lm^p$. 
Let 
 \begin{align*}
 e_{\xi} -\Phi_{\tau}(e_{\xi})
 &=\dfrac1{N^2} \sum_{i, j=1}^N (1 - \nu_1^{-i})(1- \nu_2^{-j}) [\e_{1, \tau}^{(i, j)}(\lm) \omega_{i, j}+\e_{2, \tau}^{(i, j)}(\lm) \eta_{i, j}] 
  \end{align*}
be defined in the same way. 
In this subsection, we compute $\e_{1, \tau}^{(i, N-i)}(\lm)$ and $\e_{2, \tau}^{(i, N-i)}(\lm)$ explicitly. 

The following lemma is one of the key tools to prove Theorem \ref{main:3}.

\begin{lem}[{\cite[Lemma 4.14]{As}}] \label{comp}
We have
$$\e_{1, \tau}^{(i, j)}(\lm) - \mathscr{F}^{(\sigma)}_{a_i, b_j; 1}(t)=\sum_{n=1}^{\infty} \dfrac{(t^{\tau} - t^{\sigma})^n}{n!} p^{-1} f_n(t)+b_j^{-1} \dfrac{F'_{a_i, b_j;1}(t)}{F_{a_i, b_j;1}(t)} \e_{2, \tau}^{(i, j)} (\lm), $$
where for $n \in \Z_{\geq 1}$, $f_n(t)$ is the convergent function on $\{[F_{a_i, b_j}(t)]_{<p^n}  \not \equiv 0 \pmod{p^n}\}$ defined by 
$$f_n(t)=-\dfrac{(1-\nu_1^{-i})(1-\nu_2^{-j})}{N^2} \dfrac1{F_{a_i, b_j;1}(t)} \left(\dfrac{d^{n-1}}{dt^{n-1}} \left(\dfrac{F_{a_i^{(1)}, b_j^{(1)};1}(t)}{t} \right) \right)^{\sigma}. $$
\end{lem}

From now on, we suppose that $a_i+b_j=1$, i.e. $i+j=N$. 
For simplicity, we write $H^1_{\rm dR}(X_K/S_K)^{(i, N-i)}$ as $H^1_{\rm dR}(X_K/S_K)^{(i)}$. 
Put 
\begin{align}
\begin{split} \label{w*}
&\omega^*_i :=\dfrac1{F_{a_i, 1-a_i;1}(\lm)} \omega_{i, N-i}, \\
&\eta^*_i := \lm (1- \lm) (F'_{a_i, 1-a_i;1}(\lm) \omega_{i, N-i} - (1-a_i) F_{a_i, 1-a_i;1}(\lm) \eta_{i, N-i}).  
\end{split}
\end{align}
Then they form a basis of $K((\lm)) \otimes_{\mathscr{O}(S_K)}  H^1_{\rm dR}(X_K/S_K)^{(i)}$. 

\begin{prop}[{\cite[Proposition 4.2]{As}}] \label{GM}
Let $\nabla$ be the Gauss-Manin connection on $H^1_{\rm dR}(X_K/S_K)$. 
It naturally extends to a connection on $K((\lm)) \otimes_{\mathscr{O}(S_K)} H^1_{\rm dR}(X_K/S_K)$, which we also write by $\nabla$. 
Then we have
\begin{align*}
&\begin{pmatrix}
\nabla(\omega_{i, N-i}) & \nabla(\eta_{i, N-i})
\end{pmatrix}
= 
d \lm \otimes
\begin{pmatrix}
\omega_{i, N-i} & \eta_{i, N-i}
\end{pmatrix}
\begin{pmatrix}
0 & a_i(\lm-\lm^2)^{-1} \\
1-a_i & -(1-2\lm) (\lm-\lm^2)^{-1}
\end{pmatrix}, \\
&\begin{pmatrix}
\nabla(\omega^*_{i}) & \nabla(\eta^*_{i})
\end{pmatrix}
= 
d \lm \otimes
\begin{pmatrix}
\omega^*_{i} & \eta^*_{i}
\end{pmatrix}
\begin{pmatrix}
0 & 0 \\
-\lm^{-1} (1- \lm)^{-1} F_{a_i, 1-a_i;1}(\lm)^{-2} & 0
\end{pmatrix}. 
\end{align*}
\end{prop}

Let $f \colon Y \to \mathbb{P}^1_W$ be the morphism in Lemma \ref{proj}. 
For $R=W, K$, we write $Y_R=R \times_W Y$ and $\Delta_R = \Spec R[[\lm]] \hookrightarrow \mathbb{P}^1_R$. 
Put $\mathscr{Y}_R := f^{-1}(\Delta_R)$. 
Let $D_R \subset \mathscr{Y}_R$ be the fiber at $\lm=0$ and $0=\Spec R[[\lm]]/(\lm)$. 
Then we have a commutative diagram  
\[
  \begin{CD}
     D_R @>>> \mathscr{Y}_R @>>> Y_R \\
  @VVV    @VVV  @VVV\\
    \{0\}   @>>>  \Delta_R @>>> \mathbb{P}^1_R. 
  \end{CD}
\]
We define the log de Rham complex $\omega_{\mathscr{Y}_R/R[[\lm]]}^{\bullet}$ by 
$$\omega_{\mathscr{Y}_R/R[[\lm]]}^{\bullet}= \operatorname{Coker} \left[\frac{d\lm}{\lm} \otimes \Omega^{\bullet -1}_{\mathscr{Y}_R/R[[\lm]]}(\log D_R) \to \Omega^{\bullet }_{\mathscr{Y}_R/R[[\lm]]}(\log D_R) \right]. $$

\begin{lem}
Let $H_K= H^1(\mathscr{Y}_K, \omega_{\mathscr{Y}_K/K[[\lm]]}^{\bullet})$ be Deligne's canonical extension (see \cite[(17)]{Zucker}). 
It follows from loc. cit. that $H_K \to K((\lm)) \otimes_{\mathscr{O}(S_K)}H^1_{\rm dR}(X_K/S_K)$ is injective. 
We identify $H_K$ with its image. Then the eigencomponent $H_K^{(i, N-i)}$ is a free $K[[\lm]]$-module with basis $\{\omega_i^*, \eta_i^*\}$. 
\end{lem}
\begin{proof}
Let $(H^0_K)^{(i, N-i)}= K[[\lm]] \omega_i^* \oplus K[[\lm]] \eta_i^*$. 
We can check that $H^0_K$ satisfies conditions \cite[(17)]{Zucker} by Proposition \ref{GM}. Therefore, we conclude that $H_K^{(i, N-i)}=(H_K^0)^{(i, N-i)}$ by the uniqueness of Deligne's canonical extension. 
\end{proof}

By Lemma \ref{proj}, we have 
$$D_W=N \cdot E + F, $$
where $E \cong \mathbb{P}^1$ is the exceptional divisor, and $F$ is the $N$th Fermat curve, intersecting each other transversally at $N$ points.  
Let $\mathscr{Y}_{\F}=\mathscr{Y}_W  \times_W \F$ and $D_{\F} = D_W \times_W \F$. 
Let 
$$H^{\bullet}_{\text{log-crys}}((\mathscr{Y}_{\F}, D_{\F})/(\Delta_W, 0))$$
be the log-crystalline cohomology, where $(\mathscr{X}, \mathscr{D})$ denotes the log scheme with log structure induced by the divisor $\mathscr{D}$. 
Since $E + F$ is a relative NCD over $W$ and 
$N$ is prime to $p$, the comparison theorem by Kato \cite[Theorem (6.4)]{Kato} gives
$$H^{\bullet}_{\text{log-crys}}((\mathscr{Y}_{\F}, D_{\F})/(\Delta_W, 0)) \simeq H^{\bullet} (\mathscr{Y}_W, \omega^{\bullet}_{\mathscr{Y}_W/W[[\lm]]}). $$
The log-crystalline cohomology is endowed with the $p$th Frobenius, which induces the Frobenius $\Phi_{\tau}$-action on $H_K^{(i, N-i)}$ via the isomorphism above.

\begin{thm} \label{thm:frob}
Let $j \in \{1, \ldots, N-1\}$ be the unique integer such that $pj \equiv i \pmod{N}$. 
Let $G_i^{(\tau)}(\lm) \in K[[\lm]]$ be a series defined by 
\begin{align*} 
\begin{split}
\dfrac{d}{d\lm}G_i^{(\tau)}(\lm)&=\dfrac1{\lm} \left(\dfrac{1}{(1- \lm)F_{a_i, 1-a_i;1}(\lm)^2}- \dfrac{1}{(1- \lm^p)F_{a_j, 1-a_j;1}(\lm^p)^2}\right), \\
 G_i^{(\tau)}(0)&=\widetilde{\psi}_p(a_i) + \widetilde{\psi}_p(1-a_i),  
 \end{split}
\end{align*}
where $\widetilde{\psi}_p(z)$ is a function in Definition \ref{definition:3}. 
Then we have  
$$
\begin{pmatrix}
\Phi_{\tau}(\omega^*_{j}) & \Phi_{\tau}(\eta^*_{j})
\end{pmatrix}
=
\begin{pmatrix}
\omega^*_{i}& \eta^*_{i}
\end{pmatrix}
\begin{pmatrix}
(-1)^{\frac{pj-i}{N}}p & 0 \\
(-1)^{\frac{pj-i}{N}}pG_i^{(\tau)}(\lm) & (-1)^{\frac{pj-i}{N}}
\end{pmatrix}.  
$$
\end{thm}

\begin{proof}
Since $\Phi_{\tau} \nabla=\nabla \Phi_{\tau}$, we have 
$\Phi_{\tau} \operatorname{Ker}(\nabla) \subset  \operatorname{Ker}(\nabla)$. 
Moreover, $\Phi_{\tau}$ sends the component $H_i:=K((\lm)) \otimes_{\mathscr{O}(S_K)} H^1_{\rm dR}(X_K/S_K)^{(i)}$ to the component $H_{pi}$ since $\Phi_{\tau} [\zeta, \nu]=[\zeta, \nu] \Phi_{\tau}$. 
Therefore, by Proposition \ref{GM}, we have 
$$\Phi_{\tau}(\eta_{j}^*) \in K\eta_i^*. $$
Let 
\begin{align} \label{e2}
\Phi_{\tau}(\omega_{j}^*)=pf_1(\lm) \omega_i^*+p f_2(\lm) \eta_i^*, \quad f_1(\lm), f_2(\lm) \in K[[\lm]]. 
\end{align}
By applying $\nabla$ on \eqref{e2}, we have
$$pf'_1(\lm) \omega_i^* \equiv 0 \pmod {K[[\lm]] \eta_i^*}, $$
which concludes that $f_1(\lm)$ is a constant. 
Therefore, there exist constants $\alpha_1$, $\alpha_2 \in K$ such that 
\begin{align} \label{phi1}
&\Phi_{\tau}(\omega_{j}^*)=p\alpha_1 \omega_i^* +p f_2(\lm) \eta_i^*, \\
& \Phi_{\tau}(\eta_{j}^*)=\alpha_2 \eta_i^*. 
\end{align}
Let us show $G_i^{(\tau)}(\lm)= \alpha_1 f_2(\lm)$ and $\alpha_1=\alpha_2$.  
By applying $\nabla$ on \eqref{phi1} and using Proposition \ref{GM}, we have 
$$\dfrac{d}{d\lm}f_2(\lm)=\dfrac1{ \lm} \left(\dfrac{\alpha_1}{(1- \lm)F_{a_i, 1-a_i;1}(\lm)^2}- \dfrac{\alpha_2}{(1- \lm^p)F_{a_j, 1-a_j;1}(\lm^p)^2}\right). $$
For an element $u$ of $K[[\lm]]$-module, $u|_{\lm=0}$ denotes the reduction modulo $\lm$. 
By Proposition \ref{GM}, we have 
$$\omega_{i, N-i} |_{\lm=0}=\omega_i^*|_{\lm=0}, \quad \lm \dfrac{d}{d \lm} \omega_{i, N-i}|_{\lm=0}=-\eta_i^*|_{\lm=0}. $$
By \cite[Theorem 5.4]{AH}, we have 
\begin{align*}
&\Phi_{\tau}(\omega_j^*|_{\lm=0})
=pC_i\omega_i^*|_{\lm=0} + pC_i (\widetilde{\psi}_p(a_i) +\widetilde{\psi}_p(1-a_i)) \eta_i^*|_{\lm=0}, \\
&\Phi_{\tau}(\eta_j^* |_{\lm=0})=C_i \eta_i^*|_{\lm=0}
\end{align*}
for some $C_i \in K^*$, which concludes that $\alpha_1=\alpha_2=C_i$ and $f_2(0)=C_i(\widetilde{\psi}_p(a_i) +\widetilde{\psi}_p(1-a_i))$, i.e. $f_2(\lm)= C_i G_i^{(\tau)}(\lm)$. 
Therefore, we obtain the theorem up to the constant $C_i$. 
To determine the constant, we use the Poincar\'e residue map as follows. 
\end{proof}

Let $U_1, U_2 \subset \mathscr{Y}_W$ be two affine open sets defined by 
\begin{align*}
&U_1=\Spec W[x, t]/(1+t^N-x^Nt^N), \\
 &U_2=\Spec W[y, s]/(1+s^N-y^Ns^N), 
\end{align*}
where $x=ys$ and $y=xt$.  
For $\nu \in \mu_{2N}$, let $P_{\nu}$ be the point of $U_2$ defined by $(y, s)=(0, \nu)$. 
Then the intersection locus $Z:=E \cap F$ of $D_W$ is given by  
$$Z=\{P_{\nu} \mid \nu \in \mu_{2N}, \ \nu^N=-1\}. $$
We consider the composition of morphisms 
$$\omega^{\bullet}_{\mathscr{Y}_W/W[[\lm]]} \xrightarrow{\wedge \frac{d\lm}{\lm}} \Omega^{\bullet+1}_{\mathscr{Y}_W/W} (\log D_W) \xrightarrow{\rm Res} \mathscr{O}(Z) [-1]$$
of complexes where Res is the Poincar\'e residue map.   
This gives rise to the map 
$$R \colon H_K =H^1(\mathscr{Y}_K, \omega^{\bullet}_{\mathscr{Y}_K/K[[\lm]]}) \to H^0(Z_K, \mathscr{O}(Z_K))  =\bigoplus_{P \in Z_W} K \cdot [P]. $$
This map is compatible with respect to the Frobenius $\Phi_{\tau}$ on the left and the Frobenius $\Phi_Z$ on the right in the sense that 
$$R \circ \Phi_{\tau} = p \Phi_Z \circ R. $$
 Note that $\Phi_Z$ is an $F$-linear map such that $\Phi_Z([P]) = [P]$ for any $P \in Z_W$, where $F$ is the Frobenius on $W$. 
 
\begin{proof}[Proof of Theorem \ref{thm:frob} (continued)] 
 Since $R(\lm \omega_{i, N-i})=0$ and 
 $$R(\lm \eta_{i, N-i})={\rm Res}\left(N(1-\lm)^{-1} x^{i-N} y^{j-N-1} dy \wedge d\lm\right)=0, $$
 we have $R(\eta_i^*)=0$. 
On the other hand, 
\begin{align}
\begin{split} \label{Rw}
R(\omega_i^*)= R(\omega_i)&=
{\rm Res} \left(N\dfrac{x^{i-1}y^{j-N}}{1-x^N}dx \wedge \dfrac{d \lm}{\lm} \right) \\
&={\rm Res} \left(N\dfrac{x^{i-1}y^{j-N}}{1-x^N}dx \wedge N\dfrac{(1-x^N)y^{N-1}}{x^N+y^N-x^Ny^N} dy\right) \\
&={\rm Res} \left(N^2\dfrac{s^{i-1}y^{-1}}{1+s^N-s^Ny^N} ds \wedge dy \right) \\
&={\rm Res} \left(N^2\dfrac{s^{i-1}y^{-1}}{1+s^N} \sum_{n=0}^{\infty} \left(\dfrac{s^Ny^N}{1+s^N} \right)^n ds \wedge dy \right) \\
&=-N \sum_{\substack{\nu \in \mu_{2N} \\ \nu^N=-1}} \nu^i \cdot [P_{\nu}]. 
\end{split}
\end{align}
We turn to the proof. 
Applying to $R$ on both sides of 
$$\Phi_{\tau} (\omega_j^*)=C_i (p \omega_i^*+p G_i^{(\tau)}(\lm) \eta_i^*), $$ 
we have 
$$p \Phi_Z \circ R (\omega_j^*)=C_i pR(\omega_i^*).$$
By \eqref{Rw}, we have 
$$C_i  \left(-N\sum_{\substack{\nu \in \mu_{2N} \\ \nu^N=-1}}  \nu^i \cdot [P_{\nu}]\right) =  \Phi_Z \left(-N\sum_{\substack{\nu \in \mu_{2N} \\ \nu^N=-1}}  \nu^j \cdot  [P_{\nu}]\right)=  - N\sum_{\substack{\nu \in \mu_{2N} \\ \nu^N=-1}}  \nu^{pj}\cdot  [P_{\nu}], $$
and hence $C_i=\nu^{pj-i}=(-1)^{\frac{pj-i}{N}}$. 
\end{proof}

Let 
\begin{align*}
p_i \colon \mathscr{O}(S_K)^{\dagger} \otimes_{\mathscr{O}(S_K)} H^1_{\rm dR} (X_K/S_K) \to \mathscr{O}(S_K)^{\dagger} \otimes_{\mathscr{O}(S_K)} H^1_{\rm dR} (X_K/S_K)^{(i)} \end{align*}
be the natural projection. 
Define $\e_{k, \tau}^{(i)}(\lm)$ and $E_{k, \tau}^{(i)}(\lm)$ for $k=1, 2$ by 
 \begin{align}
 \begin{split}
 p_i(e_{\xi} -\Phi_{\tau}(e_{\xi}))
 &=\dfrac1{N^2}  (1 - \nu_1^{-i})(1- \nu_2^{-(N-i)}) [\e_{1, \tau}^{(i)}(\lm) \omega_{i, N-i}+\e_{2, \tau}^{(i)}(\lm) \eta_{i, N-i}] \\
 &=\dfrac1{N^2}  (1 - \nu_1^{-i})(1- \nu_2^{-(N-i)}) [E_{1, \tau}^{(i)}(\lm)\omega_i^*+E_{2, \tau}^{(i)}(\lm)\eta_i^*].  \label{ext}
 \end{split}
 \end{align}
By \eqref{w*}, one can show that 
\begin{align}
 \e^{(i)}_{1, \tau}(\lm ) &=F_{a_i, 1-a_i;1}(\lm)^{-1}  E^{(i)}_{1, \tau}(\lm ) + (\lm - \lm^2) F'_{a_i, 1-a_i;1} (\lm) E_{2, \tau}^{(i)}(\lm),  \label{ve1}\\
  \e^{(i)}_{2, \tau}(\lm ) &= -(1-a_i)(\lm - \lm^2) F_{a_i, 1-a_i;1}(\lm) E^{(i)}_{2, \tau} (\lm). \label{ve2} 
\end{align}

\begin{thm} \label{de}
Let $\xi=\xi(\nu_1, \nu_2)$ be the Milnor symbol as in \eqref{xi}. 
 Then for any $i \in \{1, \ldots, N-1\}$, $E_{k, \tau}^{(i)}(\lm) \in K[[\lm]]$ $(k=1, 2)$ and they satisfy  
 \begin{align*}
&\dfrac{d}{d\lm}E_{1, \tau}^{(i)}(\lm)=  \dfrac{F_{a_i, 1-a_i;1}(\lm)}{1-\lm}  - (-1)^{\frac{pj-i}{N}}p^{-1} \dfrac{F_{a_j, 1-a_j;1}(\lm^p)}{1-\lm^p} \dfrac{d\lm^p}{d\lm}, \\
&\dfrac{d}{d\lm}E_{2, \tau}^{(i)}(\lm)= \dfrac{E_{1, \tau}^{(i)}(\lm) }{\lm(1-\lm)F_{a_i, 1-a_i;1}(\lm)^{2} }-(-1)^{\frac{pj-i}{N}}p^{-1}  \dfrac{F_{a_j, 1-a_j;1}(\lm^p)}{1-\lm^p}G_i^{(\tau)} (\lm) \dfrac{d\lm^p}{d\lm}, 
\end{align*} 
where $j \in \{1, \ldots, N-1\}$ is the unique integer such that $pj \equiv i \pmod{N}$. 
Moreover, we have $E_{1, \tau}^{(i)}(0)=0$. 
\end{thm}

\begin{proof}
Put $K_i=N^{-2} (1- \nu_1^{-i})(1-\nu_2^{-(N-i)})$. 
Apply $\nabla$ on \eqref{ext}. 
By Proposition \ref{GM} and using the fact that $\Phi_{\tau} \nabla= \nabla \Phi_{\tau}$, we have
\begin{align} \label{Pe}
p_i((1-\Phi_{\tau})(\nabla(e_{\xi})))=K_i \left(dE_{1, \tau}^{(i)} (\lm) \otimes\omega_i^* +\left(- \dfrac{E_{1, \tau}^{(i)}(\lm) d \lm }{\lm(1-\lm)F_{a_i, 1-a_i;1}(\lm)^2} +dE_{2, \tau}(\lm) \right) \otimes\eta_i^* \right). 
\end{align}
By \cite[(2.30)]{AM} and \cite[(4.25)]{As}, we have 
$$\nabla(e_{\xi}) =-d\log(\xi)=\dfrac1{N^2} \sum_{i, j=1}^{N} (1- \nu_1^{-i})(1-\nu_2^{-j}) \dfrac{d\lm}{1-\lm} \wedge \omega_{i, j}. $$
Therefore, the left-hand side of \eqref{Pe} is 
\begin{align*}
&K_i F_{a_i, 1-a_i;1}(\lm)\dfrac{d\lm}{1-\lm} \otimes\omega_i^* - p^{-2}  \sigma \left(K_j F_{a_j, 1-a_j;1}(\lm) \dfrac{d\lm}{1-\lm} \right) \otimes  \Phi_{\tau} (\omega_j^*)  \\
&=K_i F_{a_i, 1-a_i;1}(\lm)\dfrac{d\lm}{1-\lm} \otimes \omega_i^* - K_i p^{-1} (-1)^{\frac{pj-i}{N}}F_{a_j, 1-a_j;1}(\lm^p) \dfrac{d\lm^p}{1-\lm^p} \otimes  \left(\omega_i^* +G_i^{(\tau)} (\lm)\eta_i^*\right) \\
&=K_i \left(F_{a_i, 1-a_i;1}(\lm)\dfrac{d\lm}{1-\lm}  - p^{-1}(-1)^{\frac{pj-i}{N}} F_{a_j, 1-a_j;1}(\lm^p) \dfrac{d\lm^p}{1-\lm^p}  \right)\otimes \omega_i^*  
\\
 &  \qquad  \qquad - K_i(-1)^{\frac{pj-i}{N}}F_{a_j, 1-a_j;1}(\lm^p)G_i^{(\tau)} (\lm) \dfrac{p^{-1}d\lm^p}{1-\lm^p} \otimes\eta_i^* . 
\end{align*}
Therefore, we have 
\begin{align}
\dfrac{d}{d\lm}E_{1, \tau}^{(i)}(\lm)=  \dfrac{F_{a_i, 1-a_i;1}(\lm)}{1-\lm}  - (-1)^{\frac{pj-i}{N}}p^{-1} \dfrac{F_{a_j, 1-a_j;1}(\lm^p)}{1-\lm^p} \dfrac{d\lm^p}{d\lm}, \label{eq1}
\end{align}
and 
\begin{align}
\dfrac{d}{d\lm}E_{2, \tau}^{(i)}(\lm)= \dfrac{E_{1, \tau}^{(i)}(\lm) }{\lm(1-\lm)F_{a_i, 1-a_i;1}(\lm)^{2} }-(-1)^{\frac{pj-i}{N}}p^{-1}\dfrac{F_{a_j, 1-a_j;1}(\lm^p)}{1-\lm^p}G_i^{(\tau)} (\lm) \dfrac{d\lm^p}{d\lm} \label{eq2}. 
\end{align}
The differential equation \eqref{eq1} implies that $E_{1, \tau}^{(i)}(\lm) \in K[[\lm]]$. 
Since $G_i^{(\tau)}(\lm) \in K[[\lm]]$, by taking the residue at $\lm=0$ of both sides of \eqref{eq2}, one concludes that 
$$E_{1, \tau}^{(i)} (0)=0. $$
Since $E_{1, \tau} ^{(i)} (\lm) \in \lm K[[\lm]]$, the differential equation of \eqref{eq2} implies that $E_{2, \tau}^{(i)} (\lm) \in K[[\lm]]$, which finishes the proof. 
\end{proof}

\begin{rmk}
We have a complete description of the constant $E_{1, \tau}^{(i)}(\lm)$, while the constant $E_{2, \tau}^{(i)}(\lm)$ seems more delicate. 
\end{rmk}

\subsection{Proof of Theorem \ref{main:3}} \label{proof3} 
\begin{proof}[Proof of Theorem \ref{main:3}]
It follows from \eqref{ve1}, \eqref{ve2} and Theorem \ref{de} that 
$$\operatorname{ord}_{\lm =0} \e_{1, \tau}^{(i)}(\lm) \geq 1, \quad \operatorname{ord}_{\lm =0} \e_{2, \tau}^{(i)}(\lm) \geq 1. $$
By \cite[Lemma 4.12 (1)]{As}, $|h_{a_i, 1-a_i}(1)|_p=1$ if and only if $[F_{a_i, 1-a_i;1}(t)]_{<p^n}|_{t=1} \not \equiv 0 \pmod{p}$ for all $n \geq 1$, hence 
the function $f_n(t)$ in Lemma \ref{comp} and $F'_{a_i, 1-a_i;1}(t)/F_{a_i, 1-a_i;1}(t)$ converges at $t=1$ (see \cite[Lemma 3.4 (ii)]{Dw}).   
Therefore, by Lemma \ref{comp}, we have 
$$ \mathscr{F}_{a_i, b_j; 1}^{(\sigma)} (1)=\e_{1, \tau}^{(i)}(0)=0, $$
where the most left value is defined in \eqref{value}. 
\end{proof}

\section*{Acknowledgment}
The author sincerely thanks Noriyuki Otsubo and Masanori Asakura for valuable discussions and many helpful comments. 
The author also thanks Ryo Negishi for many helpful comments. 
The author would like to express his sincere gratitude to the anonymous
referee for many helpful comments.

\section*{Funding}
This paper is a part of the outcome of research performed under Waseda University Grant for Early Career Researchers (Project number: 2025E-041).

\end{document}